\documentclass{article}       
\usepackage{graphicx}
\DeclareGraphicsExtensions{.bmp} 
\usepackage{subfigure}
\usepackage{epstopdf}
\usepackage{cite}
\usepackage[centertags]{amsmath}
\usepackage{amsfonts,amssymb,newlfont,makeidx,listings,xcolor}
\usepackage[colorlinks=true,urlcolor=blue,linkcolor=blue,pdfborder={0 0 0}]{hyperref}
\usepackage[left=20mm,top=20mm,right=20mm,bottom=20mm,twoside=false,marginparwidth=20mm,marginparsep=20mm]{geometry}

\newtheorem{theorem}{Theorem}[section]
\newtheorem{corollary}{Corollary}[section]
\newtheorem{remark}{Remark}[section]
\newtheorem{definition}{Definition}[section]
\newenvironment{proof}{\textit{Proof}:}{\hfill$\square$}
\numberwithin{equation}{section}

\begin{document}

	\title{\textbf{Lorentzian manifolds equipped with a concircularly semi-symmetric metric connection}}
	\date{}
	\author{\textbf{Miroslav D. Maksimovi\'c\footnote{Corresponding: miroslav.maksimovic@pr.ac.rs}, Milan Lj. Zlatanovi\'c, Milica R. Vu\v{c}urovi\'c}}
	\maketitle

	\begin{abstract}
		{\small Building upon previous works characterizing GRW space-times using concircular and torse-forming vectors, this paper investigates a Lorentzian manifold equipped with a concircularly semi-symmetric metric connection. We demonstrate that such a manifold reduces to a GRW space-time under specific conditions: when the generator of the observed connection is a unit timelike vector. Also, in that case, the mentioned connection becomes a semi-symmetric metric $P$-connection. The non-zero nature of the three curvature tensors and their corresponding Ricci tensors motivates an exploration of manifold symmetries. In this way, we derive necessary and sufficient conditions for the manifold to be Einstein and we prove that a perfect fluid space-time with a semi-symmetric metric $P$-connection is Ricci pseudo-symmetric manifold of constant type. Furthermore, we show that if this space-time satisfies the Einstein's field equations without the cosmological constant,  the strong energy condition is violated.

		}
		
		
	\end{abstract}

	\section{Introduction}
	
	A Lorentzian manifold is a pseudo-Riemannian manifold with a Lorentzian metric $g$ of signature $(-,+,+,\dots,+)$ (or, equivalently, $(+,+,\dots,+,-)$). This manifold has been applied in the theory of relativity and cosmology, since space-time is a four-dimensional time-oriented Lorentzian manifold.  A generalized Robertson-Walker (briefly, GRW) space-time is a special class of Lorentzian manifolds and is defined in \cite{alias1995}. 
	
	\begin{definition}
		An $n$-dimensional Lorentzian manifold ($n\geq3$) is a generalized Robertson-Walker space-time if the metric $g$ has the form
		\begin{equation*}
			\mathrm{d}s^2=g_{ij}\mathrm{d}x^i \mathrm{d}x^j = -(\mathrm{d}t)^2 + f(t)^2 g^{*}_{\mu\nu}(\vec{x}) \mathrm{d}x^{\mu}\mathrm{d}x^{\nu},
		\end{equation*}
		where $t$ is time and $g^{*}_{\mu\nu}(\vec{x})$ is the metric of a Riemannian submanifold (of dimension $(n-1)$).
	\end{definition}
	Thus, a GRW space-time can be represented by the warped product $-I \times f^2 \mathcal{M}^{*}$, where $I$ is an open interval of real line $\mathbb{R}$, $\mathcal{M}^{*}$ is a Riemannian $(n-1)$-dimensional manifold and $f>0$ is a smooth function (or scale factor).
	
	If the metric $g^{*}$ has dimension 3 and constant curvature then a GRW space-time is actually a Robertson-Walker (briefly, RW) space-time. A GRW space-time is sometimes called a RW space-time of dimension $n>4$ (see \cite{manticamolinari2017}). Thus, GRW space-times extend a RW space-time and also encompass some other space-times, such as Lorentz-Minkowski, Einstein-de Sitter, de Sitter, Friedman cosmological model, etc.
	
	On a pseudo-Riemannian manifold, a non-zero vector $X$ is called \textit{timelike}, \textit{isotropic} (\textit{null}, \textit{lightlike}) and \textit{spacelike} if $g(X,X)<0$, $g(X,X)=0$ and $g(X,X)>0$, respectively.
	
	A \textit{torse-forming vector field} $P$ is defined in \cite{yano1944} as a vector that satisfies the relation
	\begin{equation*}
		\overset{g}{\nabla}_{X} P= \omega X + \eta(X)P,
	\end{equation*}
	where $\eta$ is an arbitrary 1-form and $\omega$ is a scalar function. In special cases, when $\eta=0$ then a torse-forming vector $P$ becomes a \textit{concircular vector field in Fialkow's sense} and if the 1-form $\eta$ is closed then a torse-forming vector field $P$ becomes a \textit{concircular vector field in Yano's sense} (see \cite{chaubey2022}). 
	By observing these vectors on Lorentzian manifolds, the following theorems were proved in the papers \cite{chen2014} and \cite{manticamolinari2017}.

	\begin{theorem}\cite{chen2014} 
		A Lorentzian manifold of dimension $n\geq3$ is a GRW space-time if and only if admits a timelike concircular vector field (in Fialkow's sense).
	\end{theorem}

	\begin{theorem}\cite{manticamolinari2017} 
		A Lorentzian manifold of dimension $n\geq3$ is a GRW space-time if and only if admits a unit timelike torse-forming vector field, $\overset{g}{\nabla}\pi = \omega (g +\pi\otimes\pi)$, which is also an eigenvector of the Ricci tensor.
	\end{theorem}
	
	The study of manifold symmetries offers valuable insights into both the geometric and physical properties of these manifolds. We now introduce some specific types of manifold symmetries.
	Let $\overset{g}{\mathcal{R}}$ denote the Riemannian curvature tensor of type (0,4), and $\overset{g}{{R}}$ denote the Riemannian curvature tensor of type (1,3). For a $(0,k)$-tensor field $B$, $k\geq1$, we define the tensor fields $\overset{g}{\mathcal{R}}\cdot B$ and $Q(g,B)$ by the equations (for example, see \cite{defever1994})
	\begin{equation}\label{eq:RputaB}
		\begin{split}
			(\overset{g}{\mathcal{R}}\cdot B)(X_1,X_2,\dots,X_k;X,Y) & = (\overset{g}{R} (X,Y)\cdot B)(X_1,X_2,\dots,X_k) \\
			& = - \sum_{i=1}^{k} B(X_1,\dots,X_{i-1},\overset{g}{R} (X,Y)X_i,X_{i+1},\dots, X_k)
		\end{split}
	\end{equation}
	and 
	\begin{equation}\label{eq:QTachibana}
		\begin{split}
			Q(g,B)(X_1,X_2,\dots,X_k;X,Y) & = ((X\wedge_{g} Y)\cdot B)(X_1,X_2,\dots,X_k) \\
			& = - \sum_{i=1}^{k} B(X_1,\dots,X_{i-1},(X\wedge_{g} Y)X_i,X_{i+1},\dots, X_k),
		\end{split}
	\end{equation}
	where $(X\wedge_{g} Y)$ is endomorphism defined by
	\begin{equation*}
		(X\wedge_{g} Y)Z=g(Y,Z)X-g(X,Z)Y.
	\end{equation*}
	The tensor $Q$ is called Tachibana tensor. Several relations involving these tensors in the context of GRW space-times were investigated in \cite{arslan2014}. By Einstein manifolds (in relation to the Levi-Civita connection) we mean manifolds that satisfy equation
		\begin{equation*}
			\overset{g}{R}ic=\lambda g,
		\end{equation*}
		where $\lambda=const.$ for dimension $n>2$ and $\overset{g}{R}ic$ is Ricci tensor with respect to the Levi-Civita connection. Denoting the conformal curvature tensor of type (0,4) by $\overset{g}{\mathcal{C}}$, the following theorem holds.
	
	\begin{theorem}\cite{arslan2014}\label{thm:RgCg-CgRg}
		On any Einstein manifold, $n\geq 4$, we have
		\begin{equation*}
			\overset{g}{\mathcal{R}}\cdot	\overset{g}{\mathcal{C}} - \overset{g}{\mathcal{C}}\cdot \overset{g}{\mathcal{R}}=\frac{1}{n-1} Q(\overset{g}{R}ic, \overset{g}{\mathcal{R}}).
		\end{equation*}
		In particular, this equation holds on any Einstein GRW space-time.
	\end{theorem}
	
	A manifold satisfying $\overset{g}{\nabla}_{X}\overset{g}{\nabla}_{Y}\overset{g}{\mathcal{R}}-  \overset{g}{\nabla}_{Y}\overset{g}{\nabla}_{X}\overset{g}{\mathcal{R}}=0$, i.e. $\overset{g}{\mathcal{R}}\cdot \overset{g}{\mathcal{R}}=0$, is called \textit{semi-symmetric}. If, instead, it satisfies $\overset{g}{\mathcal{R}}\cdot \overset{g}{\mathcal{R}}=f_1Q(g,\overset{g}{\mathcal{R}})$, it is called \textit{pseudo-symmetric manifold}, where $f_1$ is a function on the set $\mathcal{U}_1=\{x \in \mathcal{M}| Q(g,\overset{g}{\mathcal{R}})(x)\ne 0\}$. It is important to note that the term "pseudo-symmetric manifold" is used in \cite{chaki1987} to describe manifolds satisfying the condition
	\begin{equation*}
		\begin{split}
			\overset{g}{\nabla}_{X}\overset{g}{\mathcal{R}} (X_1,X_2,X_3,X_4)= & 2\eta(X)\overset{g}{\mathcal{R}} (X_1,X_2,X_3,X_4) + \eta(X_1)\overset{g}{\mathcal{R}} (X,X_2,X_3,X_4) \\
			& + \eta(X_2)\overset{g}{\mathcal{R}} (X_1,X,X_3,X_4) + \eta(X_3)\overset{g}{\mathcal{R}} (X_1,X_2,X,X_4) \\
			& + \eta(X_4)\overset{g}{\mathcal{R}} (X_1,X_2,X_3,X),
	\end{split}\end{equation*}
	where $\eta$ is an arbitrary covector. This definition is not equivalent to the previous one. The relationship between these two distinct notions of pseudo-symmetry is discussed in \cite{shaikh2015}.
	
	A manifold satisfying $\overset{g}{\nabla}_{X}\overset{g}{\nabla}_{Y}\overset{g}{R}ic-  \overset{g}{\nabla}_{Y}\overset{g}{\nabla}_{X}\overset{g}{R}ic=0$, i.e. $\overset{g}{\mathcal{R}}\cdot \overset{g}{R}ic=0$, is called \textit{Ricci semi-symmetric}. If, instead, it satisfies $\overset{g}{\mathcal{R}}\cdot \overset{g}{R}ic=fQ(g,\overset{g}{R}ic)$, it is called \textit{Ricci pseudo-symmetric manifold}, where $f_2$ is a function on the set $\mathcal{U}_2=\{x \in \mathcal{M}| Q(g,\overset{g}{R}ic)(x)\ne 0\}$. When the function $f_2=const.$ the manifold is said to be a \textit{Ricci pseudo-symmetric manifold of constant type}. 
	
	Every pseudo-symmetric manifold is also Ricci pseudo-symmetric, but the converse does not hold \cite{deszcz1989}. Every Ricci semi-symmetric manifold is also Ricci pseudo-symmetric \cite{deszcz1989}. The relationships between these classes of manifolds of dimension $n\geq4$ are illustrated in the following diagram (see \cite{deszcz2023}):
	
	\vskip0.5cm
	\begin{tabular}{ccccc}
		\framebox[4.5cm]{$\overset{g}{\mathcal{R}}\cdot \overset{g}{R}ic = f_2 Q(g,\overset{g}{R}ic)$} & $\supset$  &  \framebox[4.5cm]{$\overset{g}{\mathcal{R}}\cdot \overset{g}{\mathcal{R}} = f_1 Q(g,\overset{g}{\mathcal{R}})$} & $\subset$ & \framebox[4.5cm]{$\overset{g}{\mathcal{R}}\cdot \overset{g}{\mathcal{C}} = f_3 Q(g,\overset{g}{\mathcal{C}})$}  \\
		$\cup$ &  & $\cup$ &  & $\cup$ \\
		\framebox[4.5cm]{$\overset{g}{\mathcal{R}}\cdot \overset{g}{R}ic = 0$} & $\supset$  &  \framebox[4.5cm]{$\overset{g}{\mathcal{R}}\cdot \overset{g}{\mathcal{R}} = 0$} & $\subset$ & \framebox[4.5cm]{$\overset{g}{\mathcal{R}}\cdot \overset{g}{\mathcal{C}} = 0$}  \\
		$\cup$ &  & $\cup$ &  & $\cup$ \\
		\framebox[4.5cm]{$\overset{g}{\nabla} \overset{g}{R}ic = 0$} & $\supset$  &  \framebox[4.5cm]{$\overset{g}{\nabla} \overset{g}{\mathcal{R}} = 0$} & $\subset$ & \framebox[4.5cm]{$\overset{g}{\nabla} \overset{g}{\mathcal{C}} = 0$}  \\
		$\cup$ &  & $\cup$ &  & $\cup$ \\
		\framebox[4.5cm]{$ \overset{g}{R}ic = \frac{\overset{g}{r}}{n} g$} & $\supset$  &  \framebox[4.5cm]{$\overset{g}{\mathcal{R}} = \frac{\overset{g}{r}}{2n(n-1)}g\wedge g$} & $\subset$ & \framebox[4.5cm]{$ \overset{g}{\mathcal{C}} = 0$}
	\end{tabular}\label{dijagram1}
	\vskip0.5cm
	
 In the early 1920s, E. Cartan introduced the torsion tensor and studied it in the theory of relativity \cite{cartan1923}. After that, A. Friedmann and J. A. Schouten \cite{friedmann1924} defined a semi-symmetric connection, which is also known as a connection with vectorial torsion (for example, see \cite{agricola2016}). In paper \cite{lehel2024}, the authors described the obtaining of some types of semi-symmetric connections using symmetric connections, gave their geometric properties and cosmological application. The study of semi-symmetric connections on Lorentzian manifolds has garnered significant attention in recent years \cite{li2023Ln,ucd2024,yilmaz2023,chaubeysuhde2020, chaubey2021, li2024Ln,chaubey2022b,suh2024}. This paper focuses on Lorentzian manifolds with a concircularly semi-symmetric metric connection $\overset{1}{\nabla}$. 
	
		In addition to the connection $\overset{1}{\nabla}$, we will also study its mutual and symmetric connection.	The mutual connection $\overset{2}{\nabla}$ of $\overset{1}{\nabla}$ is uniquely defined by
		\begin{equation}\label{eq:tdualconnectionnabla2} 
			\overset{2}{\nabla}_{X} Y = \overset{1}{\nabla}_{Y} X + [X,Y],\quad \text{or equivalently}\quad  \overset{2}{\nabla}_{X} Y = \overset{1}{\nabla}_{X} Y - \overset{1}{T}(X,Y).
		\end{equation}
		In \cite{ivanov2010}, the symbols $\nabla^{+}$ and $\nabla^{-}$ are used for the connections $\overset{1}{\nabla}$ and $\overset{2}{\nabla}$ (but authors considered connections with a totally skew-symmetric torsion). In \cite{iosifidis2024}, the author used the term "torsion dual connection" for the mutual connection.
		The connection $\overset{2}{\nabla}$ has the same torsion tensor as $\overset{1}{\nabla}$ (but with the opposite sign), i.e. holds
		\begin{align*} 
			\overset{1}{T}(X,Y) & = \overset{1}{\nabla}_{X} Y - \overset{1}{\nabla}_{Y} X - [X,Y] = -(\overset{2}{\nabla}_{X} Y - \overset{2}{\nabla}_{Y} X - [X,Y]) = -\overset{2}{T}(X,Y).
		\end{align*}
		From equation (\ref{eq:tdualconnectionnabla2}) we have
		\begin{equation*}
			\overset{1}{\nabla}_{X} Y - \overset{2}{\nabla}_{Y} X - [X,Y] =\overset{2}{\nabla}_{X} Y - \overset{1}{\nabla}_{Y} X   - [X,Y]= 0,
		\end{equation*}
		which means that in the parallel transfer of vectors $X$ and $Y$ an infinitesimal parallelogram is preserved (see the Figure \ref{fig:ttandmc}). To show the significance of introducing mutual connection $\overset{2}{\nabla}$ of $\overset{1}{\nabla}$, we first explain the geometrical interpretation of the mutual connection. Let us consider two vectors $X$ and $Y$ at the same point $p$ and with endpoints $q$ and $r$, respectively. If we perform a parallel displacement of a vector $X$ along $Y$ and a parallel displacement of a vector $Y$ along $X$ with respect to the connection $\overset{1}{\nabla}$, then we obtain the vectors $X^{||1}_r$ and $Y^{||1}_q$, respectively. If the vectors $X_p$, $Y_p$, $X^{||1}_r$ and $Y^{||1}_q$ do not form an infinitesimal parallelogram, then the infinitesimal "parallelogram" is closed using the torsion tensor $\overset{1}{T}$ (actually, in this case we have a pentagon). On the other hand, if we perform a parallel displacement of $X_p$ along $Y_p$ with respect to the connection $\overset{1}{\nabla}$ (see a vector $X^{||1}_r$), and a parallel displacement of $Y_p$ along $X_p$ with respect to the connection $\overset{2}{\nabla}$ (see a vector $Y^{||2}_q$) or, vice versa, a parallel displacement of $X_p$ along $Y_p$ with respect to the connection $\overset{2}{\nabla}$ (see a vector $X^{||2}_r$), and a parallel displacement of $Y_p$ along $X_p$ with respect to the connection $\overset{1}{\nabla}$ (see a vector $Y^{||1}_q$), then the infinitesimal parallelogram ($X_pY_pX^{||1}_rY^{||2}_q$ or $X_pY_pX^{||2}_rY^{||1}_q$, respectively) is closed \cite{graiff1952}.

		\begin{figure}[ht]
			\centering
			\includegraphics[width=9cm, height=8.8cm]{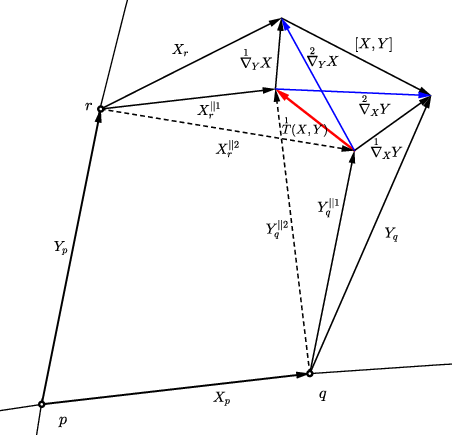}
			\caption{Geometrical interpretation of the torsion tensor $\overset{1}{T}$ and the mutual connection  $\overset{2}{\nabla}.$}
			\label{fig:ttandmc}
		\end{figure}
		
		Using the connection $\overset{1}{\nabla}$ and its mutual connection $\overset{2}{\nabla}$, the associated symmetric connection $\overset{0}{\nabla}$ is defined by the following equation
		\begin{equation*} 
			\overset{0}{\nabla}_{X} Y = \frac{1}{2} (\overset{1}{\nabla}_{X} Y + \overset{2}{\nabla}_{X} Y).
		\end{equation*}
		
		In addition to using the mutual connection to preserve the infinitesimal parallelogram during parallel transport of vectors, another motivation for considering such a connection is that it allows us to observe six linearly independent curvature tensors which are given by the following equations
		\begin{align*}
			\overset{\beta}{R} (X,Y)Z & = \overset{\beta}{\nabla}_{X}\overset{\beta}{\nabla}_{Y} Z - \overset{\beta}{\nabla}_{Y}\overset{\beta}{\nabla}_{X} Z - \overset{\beta}{\nabla}_{[X,Y]} Z, \; \beta=0,1,2, \\
			\overset{3}{R} (X,Y)Z & = \overset{2}{\nabla}_{X}\overset{1}{\nabla}_{Y} Z - \overset{1}{\nabla}_{Y}\overset{2}{\nabla}_{X} Z + \overset{2}{\nabla}_{\overset{1}{\nabla}_{Y} X} Z - \overset{1}{\nabla}_{\overset{2}{\nabla}_{X} Y} Z, \\
			\overset{4}{R} (X,Y)Z & = \overset{2}{\nabla}_{X}\overset{1}{\nabla}_{Y} Z - \overset{1}{\nabla}_{Y}\overset{2}{\nabla}_{X} Z + \overset{2}{\nabla}_{\overset{2}{\nabla}_{Y} X} Z - \overset{1}{\nabla}_{\overset{1}{\nabla}_{X} Y} Z, \\
			\overset{5}{R} (X,Y)Z & = \frac{1}{2} ( \overset{1}{\nabla}_{X}\overset{1}{\nabla}_{Y} Z - \overset{2}{\nabla}_{Y}\overset{1}{\nabla}_{X} Z + \overset{2}{\nabla}_{X}\overset{2}{\nabla}_{Y} Z - \overset{1}{\nabla}_{Y}\overset{2}{\nabla}_{X} Z  - \overset{1}{\nabla}_{[X,Y]} Z - \overset{2}{\nabla}_{[X,Y]} Z ).
		\end{align*}

		\begin{remark}
			In our previous investigation, we used the term "dual connection" instead of "mutual connection". However, to avoid confusion, the definition of mutual connection is different from the definition of dual connection that is usually used in statistical manifolds, where two connections $\nabla$ and $\nabla^*$ are said to be \emph{dual} if they satisfy the equation
			\[
			Xg(Y,Z) = g(\nabla_X Y,Z) + g(Y,\nabla^*_X Z).
			\]
		\end{remark}

	The structure of the paper is as follows: Section \ref{section2} establishes the necessary preliminaries concerning concircularly semi-symmetric metric connection on pseudo-Riemannian manifolds. Theorem \ref{thm:CSSm-SSmP} demonstrates that this connection is more general than the semi-symmetric metric $P$-connection. The equations for the linearly independent curvature tensors are also presented. Section \ref{section3} investigates Lorentzian manifolds with a concircularly semi-symmetric metric connection. We prove that, in this setting, a concircularly semi-symmetric metric connection reduces to a semi-symmetric metric $P$-connection. As a consequence, such a manifold is shown to be a GRW space-time (Corollary \ref{cor:CSSGRW}). Section \ref{section4} examines the properties of curvature tensors in a GRW space-time with a semi-symmetric metric $P$-connection. We prove that the three curvature tensors, as well as their corresponding Ricci tensors, are non-zero (Theorem \ref{thm:Rrazlicitiodnule} and Theorem \ref{thm:Ricirazlicitiodnule}). This motivates the study of the tensors $\overset{\theta}{\mathcal{R}}\cdot \overset{\theta}{R}ic$. In particular, we derive necessary and sufficient conditions for the manifold to be Einstein: the manifold is Einstein if and only if $ \overset{\alpha}{\mathcal{R}}\cdot \overset{\alpha}{R}ic=0$, $\alpha=0,4$, holds (Theorem \ref{thm:RalphaputaRic=0}).  Section \ref{section5} considers perfect fluid space-times. We prove that a perfect fluid space-time equipped with a semi-symmetric metric $P$-connection is Ricci pseudo-symmetric manifold of constant type (Theorem \ref{thm:Ricipseudosimetricankonst}) and satisfies the relation 
	$\overset{1}{\mathcal{R}}\cdot \overset{1}{R}ic=0$. Finally, section \ref{section6} explores implications for the theory of relativity. We demonstrate that if a perfect fluid space-time $(\mathcal{M},g,\overset{1}{\nabla})$ satisfies Einstein's field equations without the cosmological constant, then the strong energy condition is violated (Theorem \ref{thm:SEC}).

	

	
	\section{Concircularly semi-symmetric metric connection}\label{section2}

	Let $(\mathcal{M},g)$ be a pseudo-Riemannian manifold (dimension $n>2$) with a \textit{semi-symmetric metric connection} \cite{pak1969}
	\begin{equation}\label{eq:SSmc}
		\overset{1}{\nabla}_{X} Y = \overset{g}{\nabla}_{X} Y + \pi(Y) X - g (X,Y) P,
	\end{equation}
	whose torsion tensor is given by 
	\begin{equation*}
		\overset{1}{T}(X,Y)= \pi(Y) X - \pi(X)Y,
	\end{equation*}
	where $\pi$ is a covector, $P$ its associated vector such that $\pi(\cdot)=g(\cdot,P)$ and $\overset{g}{\nabla}$ denotes the Levi-Civita connection. Associated \textit{symmetric connection} $\overset{0}{\nabla}$ and \textit{mutual connection} $\overset{2}{\nabla}$ are defined by (see \cite{zlatanovic2021})
	\begin{align*}
		\overset{0}{\nabla}_{X} Y & = \overset{g}{\nabla}_{X} Y + \frac{1}{2}\pi(Y) X + \frac{1}{2}\pi(X) Y - g (X,Y) P,
		\\
		\overset{2}{\nabla}_{X} Y & = \overset{g}{\nabla}_{X} Y + \pi(X) Y - g (X,Y) P.
	\end{align*}

	A concircularly semi-symmetric metric connection is a special class of a semi-symmetric metric connection. Namely, if a 1-form $\pi$ satisfies equation
	\begin{equation}\label{eq:conditionforconcircularmapping}
		(\overset{g}{\nabla}_{X} \pi )(Y) - \pi(X)\pi(Y)=\omega g(X,Y),
	\end{equation}
	where $\omega$ is an arbitrary scalar function, then such a connection is said to be a \textit{concircularly semi-symmetric metric connection}\cite{slosarska1984}.
	\begin{remark}
		A conformal mapping (i.e. angle-preserving mapping) satisfying condition (\ref{eq:conditionforconcircularmapping}) is termed a concircular mapping. Concircular mappings are characterized by their preservation of  geodesic circles \cite{yano1940}.
	\end{remark}
	

	The covariant derivative of a covector $\pi$ with respect to the semi-symmetric metric connection (\ref{eq:SSmc}) satisfies the relation
	\begin{equation*}
		(\overset{1}{\nabla}_{X} \pi )(Y) = (\overset{g}{\nabla}_{X} \pi )(Y) - \pi(X)\pi(Y) + \pi(P) g(X,Y).
	\end{equation*}
	Substituting the condition  (\ref{eq:conditionforconcircularmapping}) into this expression, for a covariant derivative with respect to a concircularly semi-symmetric metric connection of a vector $P$ we obtain
	\begin{equation}\label{eq:nabla1XP}
		\overset{1}{\nabla}_{X} P = (\omega + \pi(P))X,
	\end{equation}
	from which the following statement follows. 
	\begin{theorem}
		A vector field $P$ is concircular in Fialkow's sense with respect to a concircularly semi-symmetric metric connection.
	\end{theorem}
	
	If a vector $P$ is parallel with respect to the connection $\overset{1}{\nabla}$, i.e. if $\overset{1}{\nabla} P=0$ holds, then such a connection is called a \textit{semi-symmetric metric $P$-connection} \cite{chaubey2019}. Equation (\ref{eq:nabla1XP}) shows that $\overset{1}{\nabla} P=0$ holds if and only if $\omega + \pi(P)=0$, which means that a semi-symmetric metric $P$-connection is a special case of a concircularly semi-symmetric metric connection.
	
	\begin{theorem}\label{thm:CSSm-SSmP}
		A concircularly semi-symmetric metric connection reduces to a semi-symmetric metric $P$-connection if and only if $\omega =-g(P,P)$.
	\end{theorem}

	In this paper, we continue our investigation of the curvature tensors of a concircularly semi-symmetric metric connection, initiated in \cite{mpvz2023}. 
	
	\begin{theorem} \cite{mpvz2023}
		Let $(\mathcal{M},g,\overset{1}{\nabla})$ be a pseudo-Riemannian manifold with a concircularly semi-symmetric metric connection.
		The curvature tensors $\overset{\theta}{R}$, $\theta=0,1,2,\ldots,5$, and the Riemannian curvature tensor $\overset{g}{R}$ are related by equations
		\begin{align}
			\label{eq:R0css}
			\overset{0}{R} (X,Y)Z  = &  \overset{g}{R} (X,Y)Z  +\frac{1}{2} (3\omega + \pi(P))(g(X,Z)Y - g(Y,Z) X) - \frac{1}{4}\pi(Z)(\pi(Y)X - \pi(X)Y),
			\\
			\label{eq:R1css}
			\overset{1}{R} (X,Y)Z  = & \overset{g}{R} (X,Y)Z + (2\omega + \pi(P))(g(X,Z)Y   - g(Y,Z) X),
			\\
			\label{eq:R2css}
			\overset{2}{R} (X,Y)Z  = & \overset{g}{R} (X,Y)Z + \omega(g(X,Z)Y   - g(Y,Z) X),
			\\
			\label{eq:R3css}
			\overset{3}{R} (X,Y)Z  = & \overset{g}{R} (X,Y)Z + \omega (2g(X,Z)Y  -g(Y,X) Z - g(Y,Z)X )
			+ \pi(P)(g (X,Z)Y - g (X,Y)Z),
			\\
			\label{eq:R4css}
			\begin{split}
				\overset{4}{R} (X,Y)Z  =  & \overset{g}{R} (X,Y)Z  + \omega (2g(X,Z)Y  -g(Y,X) Z - g(Y,Z)X ) 
				\\ & + \pi(P)(g (X,Z)Y - g (X,Y)Z)- \pi(Z)(\pi(Y)X - \pi(X)Y),
			\end{split}
			\\
			\label{eq:R5css}
			\overset{5}{R} (X,Y)Z  = & \overset{g}{R} (X,Y)Z + \frac{1}{2}(3\omega + \pi(P)) (g(X,Z)Y - g(Y,Z) X)  -\frac{1}{2}\pi(Y)(\pi(Z)X - \pi(X)Z).
		\end{align}
	\end{theorem}
	Curvature tensors of type (0,4) we will denote by $\overset{\theta}{\mathcal{R}}$, i.e.
	\begin{align*}
		\overset{\theta}{\mathcal{R}} (X,Y,Z,W) & =g(\overset{\theta}{R} (X,Y)Z,W), \; \; \theta=0,1,\dots,5.
	\end{align*}	
	
	In \cite{mpvz2023}, conditions for the invariance of the Riemannian curvature tensor under some connection transformations were determined. For example, if $\overset{g}{R}$ is invariant under transformation $\overset{g}{\nabla} \rightarrow \overset{0}{\nabla}$, then $\overset{g}{R} = \overset{0}{R}$ holds, and using equation (\ref{eq:R0css}) we have
	\begin{equation*}
		\frac{1}{2} (3\omega + \pi(P))(g(X,Z)Y - g(Y,Z) X) - \frac{1}{4}\pi(Z)(\pi(Y)X - \pi(X)Y)=0.
	\end{equation*}
	Contracting this equation with respect to $X$ yields
	\begin{equation*}
		\frac{n-1}{2} (3\omega + \pi(P)) g(Y,Z)  + \frac{n-1}{4}\pi(Y)\pi(Z)=0.
	\end{equation*}
	Furthermore, if we take  $Z=P$, the previous equation takes the form\footnote{The approach used in \cite{mpvz2023} differed, leading to a result that depends on the dimension of the manifold.}
	\begin{equation*}
		\frac{3(n-1)}{4} (2\omega + \pi(P)) \pi(Y)=0,
	\end{equation*}
	from where
	\begin{equation*}
		2\omega +\pi(P)=0.
	\end{equation*}
	\begin{theorem} \label{thm:transfconnCSSM}
		Let $(\mathcal{M},g,\overset{1}{\nabla})$ be a pseudo-Riemannian manifold with a concircularly semi-symmetric metric connection $\overset{1}{\nabla}$, $\overset{0}{\nabla}$ be the symmetric connection of $\overset{1}{\nabla}$, $\overset{2}{\nabla}$ be the dual connection of $\overset{1}{\nabla}$ and $\overset{g}{\nabla}$ be the Levi-Civita connection. Then, we have:
		\begin{itemize}
			\item[1)] If the Riemannian curvature tensor $\overset{g}{R}$ is invariant under connection transformation $\overset{g}{\nabla} \rightarrow \overset{0}{\nabla}$,  then $2\omega=-g(P,P)$.
			\item[2)] The Riemannian curvature tensor $\overset{g}{R}$ is invariant under connection transformation $\overset{g}{\nabla} \rightarrow \overset{1}{\nabla}$ if and only if $2\omega=-g(P,P)$.
			\item[3)] The Riemannian curvature tensor $\overset{g}{R}$ is invariant under connection transformation $\overset{g}{\nabla} \rightarrow \overset{2}{\nabla}$ if and only if $\omega=0$.
		\end{itemize}
	\end{theorem} 
	
	By contracting equations (\ref{eq:R0css})-(\ref{eq:R5css}) with respect to the vector $X$, we obtain the corresponding Ricci tensors $\overset{\theta}{R}ic$, $\theta=0,1,\dots,5$. With $\overset{\theta}{r}$ and $\overset{g}{r}$ we denote the scalar curvatures, i.e. the trace of Ricci tensors $\overset{\theta}{R}ic$, $\theta=0,1,\dots,5$, and $\overset{g}{R}ic$.
	
	\begin{theorem}
		Let $(\mathcal{M},g,\overset{1}{\nabla})$ be a pseudo-Riemannian manifold with a concircularly semi-symmetric metric connection.
		Ricci tensors $\overset{\theta}{R}ic$, $\theta=0,1,2,\ldots,5$, and the Ricci tensor $\overset{g}{R}ic$ are related by equations
		\begin{align}
			\label{eq:Ric0css}
			\overset{0}{R}ic  = &  \overset{g}{R}ic  -\frac{n-1}{2}  (3\omega + \pi(P)) g - \frac{n-1}{4}\pi\otimes\pi,
			\\
			\label{eq:Ric1css}
			\overset{1}{R}ic = & \overset{g}{R}ic - (n-1)(2\omega +\pi(P)) g,
			\\
			\label{eq:Ricalphacss}
			\overset{\alpha}{R}ic = & \overset{g}{R}ic - (n-1)\omega g, \;\; \alpha=2,3,
			\\
			\label{eq:Ric4css}
			\overset{4}{R}ic = & \overset{g}{R}ic - (n-1) \omega  g - (n-1)\pi\otimes\pi,
			\\
			\label{eq:Ric5css}
			\overset{5}{R}ic = & \overset{g}{R}ic- \frac{n-1}{2} (3\omega + \pi(P) )  g - \frac{n-1}{2}\pi\otimes\pi.
		\end{align}
	\end{theorem}
	
	From here, we can conclude that Ricci tensors $\overset{\theta}{R}ic$, $\theta=0,1,\dots,5$, are symmetric.
	
	\begin{theorem}
		Let $(\mathcal{M},g,\overset{1}{\nabla})$ be a pseudo-Riemannian manifold with a concircularly semi-symmetric metric connection.
		Scalar curvatures  $\overset{\theta}{r}$, $\theta=0,1,2,\ldots,5$, and scalar curvature $\overset{g}{r}$ are related by equations
		\begin{align}
			\label{eq:r0css}
			\overset{0}{r}  = &  \overset{g}{r}   -\frac{3n(n-1)}{2}  \omega - \frac{(n-1)(2n+1)}{4}\pi(P),
			\\
			\label{eq:r1css1}
			\overset{1}{r}  = &  \overset{g}{r}   -2n(n-1) (\omega + \frac{1}{2}\pi(P)),
			\\
			\label{eq:ralphacss}
			\overset{\alpha}{r}  = & \overset{g}{r} - n(n-1)\omega, \;\; \alpha=2,3
			\\
			\label{eq:r4css}
			\overset{4}{r}  = & \overset{g}{r}  - n(n-1) \omega - (n-1)\pi(P),
			\\
			\label{eq:r5css}
			\overset{5}{r}= & \overset{g}{r}- \frac{3n(n-1)}{2}\omega - \frac{n^2-1}{2}\pi(P).
		\end{align}
	\end{theorem}


	\section{Lorentzian manifolds} \label{section3}
	
	We now consider the application of a concircularly semi-symmetric metric connection with a unit timelike vector $P$ to Lorentzian manifolds. Differentiating the relation $\pi(P)=g(P,P)=-1$ (with respect to the Levi-Civita connection) and using equation (\ref{eq:conditionforconcircularmapping}) we obtain $(\omega - 1)\pi(X)=0$, which implies $\omega=1$. Equation (\ref{eq:conditionforconcircularmapping}) then takes the form
	\begin{equation*}
		(\overset{g}{\nabla}_{X} \pi )(Y) = g(X,Y) + \pi(X)\pi(Y),
	\end{equation*}
	from where
	\begin{equation*}
		\overset{g}{\nabla}_{X} P= X + \pi(X)P,
	\end{equation*}
	or equivalently
	\begin{equation}\label{eq:Ptorse-forming}
		\overset{g}{\nabla}_{X} P= -\pi(P)(X + \pi(X)P).
	\end{equation}
	This demonstrates that the unit timelike vector $P$ is torse-forming. Furthermore, we see that $\omega + \pi(P)=0$ holds. Substituting this into equation (\ref{eq:nabla1XP}) yields $\overset{1}{\nabla}P=0$, indicating that the connection is a semi-symmetric metric $P$-connection. This establishes the following theorem.
	
	\begin{theorem}
		If the vector field $P$ is unit timelike, then a concircularly semi-symmetric metric connection is a semi-symmetric metric $P$-connection.
	\end{theorem}
	
	Lorentzian manifolds equipped with a semi-symmetric metric $P$-connection were studied in \cite{chaubeysuhde2020}, where the following theorem was proven. 
	
	\begin{theorem}\cite{chaubeysuhde2020} \label{thm:SSPGRW}
		Let $\mathcal{M}$ be an $n$-dimensional $(n\geq3)$ Lorentzian manifold equipped with a semi-symmetric metric $P$-connection whose associated vector $P$ is a unit timelike torse-forming vector field. Then $\mathcal{M}$ is a GRW space-time.
	\end{theorem}
	
	A semi-symmetric metric $P$-connection on Lorentzian manifolds was also investigated in \cite{yilmaz2023}, where the previous theorem was expressed in a different form. Combining the preceding two theorems, we have the following corollary.
	
	\begin{corollary}\label{cor:CSSGRW}
		An $n$-dimensional Lorentzian manifold $(n\geq3)$ equipped with a concircularly semi-symmetric metric connection whose associated vector $P$ is a unit timelike vector field is a GRW space-time.
	\end{corollary}
	
	In contrast to the aforementioned papers, \cite{ucd2024} studied a semi-symmetric metric connection without the condition that the vector $P$ be parallel with respect to the observed connection and the following theorem was established.
	
	\begin{theorem}\cite{ucd2024}\label{thm:SSRicttGRW}
		A Lorentzian manifold of dimension $n\geq3$ with a semi-symmetric metric connection whose Ricci tensor is symmetric and torsion tensor is recurrent represents a GRW space-time. 
	\end{theorem}
	
	From equation (2.18) in \cite{ucd2024}, we deduce the following theorem.
	\begin{theorem}
		A semi-symmetric metric connection with associated unit timelike vector $P$ and recurrent torsion tensor is a semi-symmetric metric $P$-connection.
	\end{theorem}
	
	Therefore, Theorem \ref{thm:SSRicttGRW} is an equivalent formulation of Theorem \ref{thm:SSPGRW} or Corollary \ref{cor:CSSGRW}. Hereinafter, by $(\mathcal{M},g,\overset{1}{\nabla})$ we will denote an $n$-dimensional GRW space-time with a semi-symmetric metric $P$-connection.

	\section{Curvature properties}\label{section4}

	Based on relations (\ref{eq:R0css}) - (\ref{eq:R5css}) we derive the equations for the curvature tensors in a GRW space-time $(\mathcal{M},g,\overset{1}{\nabla})$.
	
	\begin{theorem}
		In a GRW space-time $(\mathcal{M},g,\overset{1}{\nabla})$, curvature tensors $\overset{\theta}{R}$, $\theta=0,1,2,\ldots,5$, and the Riemannian curvature tensor $\overset{g}{R}$ are related by equations
		\begin{align}
			\label{eq:R0cssLn}
			\overset{0}{R} (X,Y)Z  = &  \overset{g}{R} (X,Y)Z  +g(X,Z)Y   - g(Y,Z) X - \frac{1}{4}\pi(Z)(\pi(Y)X - \pi(X)Y),
			\\
			\label{eq:RbetacssLn}
			\overset{\beta}{R} (X,Y)Z  = & \overset{g}{R} (X,Y)Z + g(X,Z)Y   - g(Y,Z) X, \;\; \beta=1,2,3,
			\\
			\label{eq:R4cssLn}
			\begin{split}
				\overset{4}{R} (X,Y)Z  =  & \overset{g}{R} (X,Y)Z  + g(X,Z)Y   - g(Y,Z) X - \pi(Z)(\pi(Y)X - \pi(X)Y),
			\end{split}
			\\
			\label{eq:R5cssLn}
			\overset{5}{R} (X,Y)Z  = & \overset{g}{R} (X,Y)Z + g(X,Z)Y   - g(Y,Z) X  -\frac{1}{2}\pi(Y)(\pi(Z)X - \pi(X)Z).
		\end{align}
	\end{theorem}

	\begin{theorem}
		In a GRW space-time $(\mathcal{M},g,\overset{1}{\nabla})$, Ricci tensors $\overset{\theta}{R}ic$, $\theta=0,1,2,\ldots,5$, and the Ricci tensor $\overset{g}{R}ic$ are related by equations
		\begin{align}
			\label{eq:Ric0cssLn}
			\overset{0}{R}ic  = &  \overset{g}{R}ic  -\frac{n-1}{4} (4g + \pi\otimes\pi ) ,
			\\
			\label{eq:RicbetacssLn}
			\overset{\beta}{R}ic = & \overset{g}{R}ic - (n-1) g, \;\; \beta=1,2,3,
			\\
			\label{eq:Ric4cssLn}
			\overset{4}{R}ic = & \overset{g}{R}ic - (n-1)  (g + \pi\otimes\pi ),
			\\
			\label{eq:Ric5cssLn}
			\overset{5}{R}ic = & \overset{g}{R}ic- \frac{n-1}{2} (2g + \pi\otimes\pi ).
		\end{align}
	\end{theorem}
	
	Equation (\ref{eq:RbetacssLn}) shows that the curvature tensors $\overset{1}{R}$, $\overset{2}{R}$ and $\overset{3}{R}$ coincide. Therefore, in what follows, we will use only $\overset{1}{R}$ and $\overset{1}{R}ic$ instead of $\overset{\beta}{R}$ and $\overset{\beta}{R}ic$, $\beta=1,2,3$, respectively.
	
	The parallelism of the vector $P$ with respect to the connection $\overset{1}{\nabla}$, readily yields that equations 
	\begin{align*}
		\overset{1}{R} (X,Y)P & = \overset{1}{R} (P,Y)Z=0, \\
		\pi(\overset{1}{R} (X,Y)Z) & =0, \\
		\overset{1}{R}ic (P,X) & =0,
	\end{align*}
	hold in a GRW space-time $(\mathcal{M},g,\overset{1}{\nabla})$, confirming the results from \cite{yilmaz2023}. On the other hand, the Riemannian curvature tensor $\overset{g}{R}$ and the Ricci tensor $\overset{g}{R}ic$, in a GRW space-time $(\mathcal{M},g,\overset{1}{\nabla})$, satisfy  equations
	\begin{align}\label{eq:RXYP}
		\overset{g}{R} (X,Y)P & = \pi(Y)X-\pi(X)Y= \overset{1}{T} (X,Y), \\
		\label{eq:RPYZ}
		\overset{g}{R} (P,Y)Z & = g(Y,Z)P - \pi(Z)Y, \\
		\label{eq:piRXYZ}
		\pi(\overset{g}{R} (X,Y)Z) & = \pi(X)g(Y,Z)-\pi(Y)g(X,Z)=g(\overset{1}{T} (X,Y),Z)= \overset{1}{T} (X,Y,Z), \\
		\label{eq:RicPX}
		\overset{g}{R}ic(P,X) & = (n-1)\pi(X),
	\end{align}
	which can be found in the equivalent form in \cite{chaubeysuhde2020,siddiqi2019}. Taking into account the previous equations, we can easily show that the curvature tensors $\overset{0}{R}$, $\overset{4}{R}$ and  $\overset{5}{R}$ have the following properties
	\begin{align*}
		4\overset{0}{R} (X,Y)P & = \overset{4}{R} (X,Y)P= 2\overset{5}{R} (X,P)Y = \overset{1}{T} (X,Y),\\
		4\overset{0}{R} (P,Y)Z & = \overset{4}{R} (P,Y)Z = 2\overset{5}{R} (P,Z)Y = - \pi(Z)\overset{g}{\nabla}_Y P,\\
		\overset{0}{R} (P,P)X & = \overset{4}{R} (P,P)X = \overset{5}{R} (P,X)P=0,\\
		\pi(\overset{\theta}{R} (X,Y)Z) & = 0, \theta=0,4,5,
	\end{align*}
	while the corresponding Ricci tensors satisfy the following relations
	\begin{equation*}
		4\overset{0}{R}ic(P,X) = \overset{4}{R}ic(P,X) = 2\overset{5}{R}ic(P,X) = (n-1)\pi(X).
	\end{equation*}
	
	These equations further imply that $\frac{n-1}{4}$, $(n-1)$ and $\frac{n-1}{2}$ are eigenvalues of the Ricci tensors $\overset{0}{R}ic$, $\overset{4}{R}ic$, $\overset{5}{R}ic$, respectively, corresponding to the eigenvector $P$. On the other hand, equation (\ref{eq:RicPX}) implies that $(n-1)$ is an eigenvalue of the Ricci tensor $\overset{g}{R}ic$ corresponding to the eigenvector $P$.
	
	Since $2\omega \ne -g(P,P)$, Theorem \ref{thm:transfconnCSSM} implies that the Riemannian curvature tensor $\overset{g}{R}$ is not invariant under the transformation of connections $\overset{g}{\nabla} \rightarrow \overset{1}{\nabla}$, so the following theorem is easily proven.

	\begin{theorem}
		In a GRW space-time $(\mathcal{M},g,\overset{1}{\nabla})$, the Riemannian curvature tensor $\overset{g}{R}$ cannot be invariant under the transformation of connections $\overset{g}{\nabla} \rightarrow \overset{\beta}{\nabla}$, $\beta=0,1,2$.
	\end{theorem}
	
	If the curvature tensor $\overset{1}{R}$ vanishes, then the GRW space-time $(\mathcal{M},g,\overset{1}{\nabla})$ is locally isometric to a unit sphere $S^n(1)$ (see \cite{chaubeysuhde2020}). We now show that the other curvature tensors $\overset{\theta}{R}$, $\theta=0,4,5$, cannot vanish.
	
	\begin{theorem}\label{thm:Rrazlicitiodnule}
		In a GRW space-time $(\mathcal{M},g,\overset{1}{\nabla})$, the curvature tensors $\overset{0}{R}$, $\overset{4}{R}$ and $\overset{5}{R}$ are non-zero.
	\end{theorem}
	\begin{proof}
		We provide the proof for the curvature tensor of the zero kind $\overset{0}{R}$. If $\overset{0}{R}=0$ holds, then equation (\ref{eq:R0cssLn}) gives
		\begin{equation*}
			\overset{g}{R} (X,Y)Z  +g(X,Z)Y   - g(Y,Z) X - \frac{1}{4}\pi(Z)(\pi(Y)X - \pi(X)Y)=0.
		\end{equation*}
		Taking $Z=P$, from the previous equation we have
		\begin{equation*}
			\overset{g}{R} (X,Y)P  +\pi(X)Y   - \pi(Y) X - \frac{1}{4}\pi(Z)(\pi(Y)X - \pi(X)Y)=0.
		\end{equation*}
		Using equation (\ref{eq:RXYP}), we further obtain
		\begin{equation*}
			\overset{1}{T}(X,Y) - \overset{1}{T}(X,Y) - \frac{1}{4}\pi(Z)\overset{1}{T}(X,Y)=0,
		\end{equation*}
		from where $\overset{1}{T}\otimes\pi=0$, which is impossible.
	\end{proof}
	
	Similarly, if the Ricci tensors $\overset{0}{R}ic$, $\overset{4}{R}ic$ and $\overset{5}{R}ic$ vanish, then the properties of the Ricci tensor $\overset{g}{R}ic$ are violated.
	
	\begin{theorem}\label{thm:Ricirazlicitiodnule}
		In a GRW space-time $(\mathcal{M},g,\overset{1}{\nabla})$, the Ricci tensors $\overset{0}{R}ic$, $\overset{4}{R}ic$ and $\overset{5}{R}ic$ are non-zero.
	\end{theorem}
	\begin{proof}
		For example, if the Ricci tensor $\overset{5}{R}ic$ vanishes, then equation (\ref{eq:Ric5cssLn}) implies
		\begin{equation*}
			\overset{g}{R}ic= \frac{n-1}{2} (2g + \pi\otimes\pi ).
		\end{equation*}
		From here it follows
		\begin{equation*}
			\overset{g}{R}ic(X,P)= \frac{n-1}{2} (2\pi(X) -\pi(X) ) = \frac{n-1}{2}\pi(X),
		\end{equation*}
		which contradicts equation (\ref{eq:RicPX}).
	\end{proof}
	
	Motivated by these results, we now consider weaker conditions than vanishing for the curvature tensors and Ricci tensors. Analogously to equation (\ref{eq:RputaB}), for a $(0,k)$-tensor field $B$ on $(\mathcal{M},g, \overset{1}{\nabla})$, $k\geq1$, we can define a tensor field $\overset{\theta}{\mathcal{R}}\cdot B$, $\theta\in\{0,1,4,5\}$, by equations
	\begin{equation}\label{eq:RthetaputaB}
		\begin{split}
			(\overset{\theta}{\mathcal{R}}\cdot B)(X_1,X_2,\dots,X_k;X,Y) & = (\overset{\theta}{R} (X,Y)\cdot B)(X_1,X_2,\dots,X_k) \\
			& = - \sum_{i=1}^{k} B(X_1,\dots,X_{i-1},\overset{\theta}{R} (X,Y)X_i,X_{i+1},\dots, X_k).
		\end{split}
	\end{equation}
	
	As shown in the diagram on page \pageref{dijagram1}, every Einstein manifold is Ricci semi-symmetric, but the converse is not necessarily true. We now prove the following theorem.
	
	\begin{theorem}\label{thm:GRWRicipolu-sim}
		A GRW space-time $(\mathcal{M},g,\overset{1}{\nabla})$ is a Ricci semi-symmetric manifold if and only if it is an Einstein manifold.
	\end{theorem}
	\begin{proof}
		If the manifold $(\mathcal{M},g,\overset{1}{\nabla})$ is Ricci semi-symmetric, then 
		\begin{equation*}
			\overset{g}{R}ic(\overset{g}{R} (X,Y)U,V) + \overset{g}{R}ic(U,\overset{g}{R} (X,Y)V)=0,
		\end{equation*}
		holds. Replacing the vectors $X$ and $V$ with $P$ in this equation yields
		\begin{equation*}
			\overset{g}{R}ic(\overset{g}{R} (P,Y)U,P) + \overset{g}{R}ic(U,\overset{g}{R} (P,Y)P)=0.
		\end{equation*}
		Using equation (\ref{eq:RPYZ}), we then obtain
		\begin{equation*}
			g(Y,U)\overset{g}{R}ic(P,P)-\pi(U)\overset{g}{R}ic(Y,P)+ \pi(Y)\overset{g}{R}ic(U,P)+ \overset{g}{R}ic(Z,Y)=0.
		\end{equation*}
		Substituting (\ref{eq:RicPX}) into the previous equation, after rearranging, we arrive at the following form of the Ricci tensor
		\begin{equation*}
			\overset{g}{R}ic = (n-1)g.
		\end{equation*}
		Therefore, the observed manifold is Einstein.
	\end{proof}
	
	We now examine the relations $\overset{\theta}{\mathcal{R}}\cdot \overset{\theta}{R}ic$, $\theta=0,1,4,5$, defined by equation (\ref{eq:RthetaputaB}). First, we define the following Tachibana type tensor
	\begin{equation}\label{eq:QTachibana2}
		\begin{split}
			Q(\overset{g}{R}ic,\Pi)(X_1,X_2;X,Y) & = ((X\wedge_{\overset{g}{R}ic} Y)\cdot \Pi)(X_1,X_2) \\
			& = - \Pi((X\wedge_{\overset{g}{R}ic} Y)X_1, X_2) - \Pi(X_1, (X\wedge_{\overset{g}{R}ic} Y)X_2),
		\end{split}
	\end{equation}
	where $\Pi=\pi\otimes\pi$ and $(X\wedge_{\overset{g}{R}ic} Y)$ is defined by
	\begin{equation*}
		(X\wedge_{\overset{g}{R}ic} Y)Z=\overset{g}{R}ic(Y,Z)X-\overset{g}{R}ic(X,Z)Y.
	\end{equation*}

	Now we can prove the following theorem for the curvature tensors of a semi-symmetric metric $P$-connection.
	\begin{theorem} In a GRW space-time  $(\mathcal{M},g,\overset{1}{\nabla})$, the curvature tensors $\overset{0}{R}$, $\overset{1}{R}$, $\overset{4}{R}$ and the Ricci tensors $\overset{0}{R}ic$, $\overset{1}{R}ic$, $\overset{4}{R}ic$ satisfy the following relations
		\begin{align}\label{eq:R0putaRic0}
			\overset{0}{\mathcal{R}}\cdot \overset{0}{R}ic & = \overset{g}{\mathcal{R}}\cdot \overset{g}{R}ic - Q(g,\overset{g}{R}ic) - \frac{n-1}{4}Q(g,\Pi) + \frac{1}{4}Q(\overset{g}{R}ic,\Pi), \\
			\label{eq:R1putaRic1}
			\overset{1}{\mathcal{R}}\cdot \overset{1}{R}ic & = \overset{g}{\mathcal{R}}\cdot \overset{g}{R}ic - Q(g,\overset{g}{R}ic), \\
			\label{eq:R4putaRic4}
			\overset{4}{\mathcal{R}}\cdot \overset{4}{R}ic & = \overset{g}{\mathcal{R}}\cdot \overset{g}{R}ic - Q(g,\overset{g}{R}ic) - (n-1)Q(g,\Pi) + Q(\overset{g}{R}ic,\Pi),
		\end{align}
		where $\Pi=\pi\otimes\pi$, and $Q(\cdot,\cdot)$ are Tachibana type tensors defined by equations (\ref{eq:QTachibana}) and (\ref{eq:QTachibana2}), respectively.
	\end{theorem}
	\begin{proof}
		We present the proof for the last equation. Starting with
		\begin{equation*}
			(\overset{4}{R} (X,Y)\cdot \overset{4}{R}ic)(U,V) = - \overset{4}{R}ic(\overset{4}{R} (X,Y)U,V) - \overset{4}{R}ic(U,\overset{4}{R} (X,Y)V),
		\end{equation*}
		and using equations (\ref{eq:R4cssLn}) and (\ref{eq:Ric4cssLn}), we obtain
		\begin{equation}\label{eq:R4putaRic42}
			\begin{split}
				(\overset{4}{R} (X,Y)\cdot \overset{4}{R}ic)(U,V) = &  - \overset{g}{R}ic(\overset{g}{R} (X,Y)U,V) - \overset{g}{R}ic(U,\overset{g}{R} (X,Y)V) - (g(X,U) + \pi(X)\pi(U)) \overset{g}{R}ic(Y,V) \\
				& + (g(Y,U) + \pi(Y)\pi(U)) \overset{g}{R}ic(X,V) - (g(X,V) + \pi(X)\pi(V)) \overset{g}{R}ic(U,Y) \\
				& + (g(Y,V) + \pi(Y)\pi(V)) \overset{g}{R}ic(U,X) + (n-1) \pi(V)(\pi(X)g(Y,U) - \pi(Y)g(X,U)) \\
				& + (n-1) \pi(U)(\pi(X)g(Y,V) - \pi(Y)g(X,V) ),
			\end{split}
		\end{equation}
		from which it arises
		\begin{equation*}
			\overset{4}{\mathcal{R}}\cdot \overset{4}{R}ic  = \overset{g}{\mathcal{R}}\cdot \overset{g}{R}ic - Q(g,\overset{g}{R}ic) - (n-1)Q(g,\Pi) + Q(\overset{g}{R}ic,\Pi).
		\end{equation*}
	\end{proof}
	
	Based on the previous theorem, we have the following corollary.
	\begin{corollary}
		In a GRW space-time  $(\mathcal{M},g,\overset{1}{\nabla})$, the curvature tensor $\overset{1}{\mathcal{R}}$ and the Ricci tensor $\overset{1}{R}ic$ satisfy the relation $ \overset{1}{\mathcal{R}}\cdot \overset{1}{R}ic=0$ if and only if the manifold is Ricci pseudo-symmetric of constant type, whose the form is given by
		\begin{equation*}
			\overset{g}{\mathcal{R}}\cdot \overset{g}{R}ic = Q(g,\overset{g}{R}ic).
		\end{equation*}
	\end{corollary}
	
	It was shown in \cite{chaubeysuhde2020} that $\overset{1}{\mathcal{R}}\cdot \overset{1}{R}ic= \overset{g}{\mathcal{R}}\cdot \overset{g}{R}ic$ if and only if the manifold is Ricci flat with respect to the semi-symmetric metric $P$-connection  $\overset{1}{\nabla}$ or the manifold is Einstein with respect to the metric $g$. We now investigate the implications of the tensors $ \overset{0}{\mathcal{R}}\cdot \overset{0}{R}ic$ and $ \overset{4}{\mathcal{R}}\cdot \overset{4}{R}ic$.

	\begin{theorem}\label{thm:RalphaputaRic=0}
		A GRW space-time  $(\mathcal{M},g,\overset{1}{\nabla})$ is an Einstein manifold if and only if $ \overset{\alpha}{\mathcal{R}}\cdot \overset{\alpha}{R}ic=0$, $\alpha=0,4$, holds.
	\end{theorem}
	\begin{proof}
		If $\overset{4}{\mathcal{R}}\cdot \overset{4}{R}ic=0$ holds, then from equation (\ref{eq:R4putaRic42}) and setting $X=V=P$, we get
		\begin{equation*}
			\begin{split}
				- \overset{g}{R}ic(\overset{g}{R} (P,Y)U,P) & - \overset{g}{R}ic(U,\overset{g}{R} (P,Y)P) - (g(P,U) + \pi(P)\pi(U)) \overset{g}{R}ic(Y,P) \\
				& + (g(Y,U) + \pi(Y)\pi(U)) \overset{g}{R}ic(P,P) - (g(P,P) + \pi(P)\pi(P)) \overset{g}{R}ic(U,Y) \\
				& + (g(Y,P) + \pi(Y)\pi(P)) \overset{g}{R}ic(U,P) + (n-1) \pi(P)(\pi(P)g(Y,U) - \pi(Y)g(P,U)) \\
				& + (n-1) \pi(U)(\pi(P)g(Y,P) - \pi(Y)g(P,P) ) = 0.
			\end{split}
		\end{equation*}
		Using the properties of the curvature tensor $\overset{g}{R}$ and Ricci tensor $\overset{g}{R}ic$ (specifically, equations (\ref{eq:RPYZ}) and (\ref{eq:RicPX})), after rearranging, we obtain
		\begin{equation}\label{eq:Ajnstajn}
			\overset{g}{R}ic = (n-1)g.
		\end{equation}
		This implies that such a manifold is Einstein. Conversely, if (\ref{eq:Ajnstajn}) holds, then (\ref{eq:R4putaRic42}) implies $\overset{4}{\mathcal{R}}\cdot \overset{4}{R}ic=0$.
	\end{proof}
	
	For the curvature tensor $\overset{5}{\mathcal{R}}$ and Ricci tensor $\overset{5}{R}ic$ we have the following statement.
	
	\begin{theorem}\label{thm:R5Ric5=RgRicg}
		A GRW space-time	$(\mathcal{M},g,\overset{1}{\nabla})$ is an Einstein manifold if and only if $ \overset{5}{\mathcal{R}}\cdot \overset{5}{R}ic $ $=\overset{g}{\mathcal{R}}\cdot \overset{g}{R}ic$.
	\end{theorem}
	\begin{proof}
		For the tensor field $ \overset{5}{\mathcal{R}}\cdot \overset{5}{R}ic$ we have
		\begin{equation}\label{eq:R5putaRic5}
			\begin{split}
				(\overset{5}{R} (X,Y)\cdot \overset{5}{R}ic)(U,V) = &   (\overset{g}{R} (X,Y)\cdot \overset{g}{R}ic)(U,V) - \frac{n-1}{2}\pi(Y) ( \pi(V)g(X,U) +  \pi(U)g(X,V)) \\
				& + \frac{1}{2}\pi(Y) (\pi(V)\overset{g}{R}ic(U,X) + \pi(U)\overset{g}{R}ic(X,V)) \\
				& - \pi(Y)\pi(X) (\overset{g}{R}ic(U,V) - (n-1)g(U,V)) - g(X,U)\overset{g}{R}ic(Y,V) \\
				& + g(Y,U)\overset{g}{R}ic(X,V) - g(X,V)\overset{g}{R}ic(U,Y)+ g(Y,V)\overset{g}{R}ic(U,X).
			\end{split}
		\end{equation}
		Assume that $\overset{5}{\mathcal{R}}\cdot \overset{5}{R}ic=\overset{g}{\mathcal{R}}\cdot \overset{g}{R}ic$. Setting $X=V=P$, and using the equations (\ref{eq:RPYZ}) and (\ref{eq:RicPX}), we find, after rearranging, that the manifold is Einstein, with the Ricci tensor of the form (\ref{eq:Ajnstajn}). 
		
		Conversely, if we assume that (\ref{eq:Ajnstajn}) holds, then substituting this equation into (\ref{eq:R5putaRic5}) gives $\overset{5}{\mathcal{R}}\cdot \overset{5}{R}ic=\overset{g}{\mathcal{R}}\cdot \overset{g}{R}ic$.
	\end{proof}
	
	Since the following equation holds in every pseudo-Riemannian Einstein manifold (of dimension $n\geq4$) (see Theorem 3.1 in \cite{deszcz2001}) 
	\begin{equation*}
		\overset{g}{\mathcal{R}}\cdot	\overset{g}{\mathcal{C}} - \overset{g}{\mathcal{C}}\cdot \overset{g}{\mathcal{R}} = \frac{\overset{g}{r}}{n(n-1)}Q(g, \overset{g}{\mathcal{R}}) =\frac{\overset{g}{r}}{n(n-1)} Q(g, \overset{g}{\mathcal{C}}),
	\end{equation*}
	Theorems \ref{thm:RgCg-CgRg}, \ref{thm:RalphaputaRic=0} and \ref{thm:R5Ric5=RgRicg}, lead to the following corollary.
	\begin{corollary}
		In a GRW space-time $(\mathcal{M},g,\overset{1}{\nabla})$, dimension $n\geq4$, if either $ \overset{\alpha}{\mathcal{R}}\cdot \overset{\alpha}{R}ic=0$, $\alpha=0,4$, or $ \overset{5}{\mathcal{R}}\cdot \overset{5}{R}ic $ $=\overset{g}{\mathcal{R}}\cdot \overset{g}{R}ic$ holds, then the following relation
		\begin{equation*}
			\overset{g}{\mathcal{R}}\cdot	\overset{g}{\mathcal{C}} - \overset{g}{\mathcal{C}}\cdot \overset{g}{\mathcal{R}}=\frac{1}{n-1} Q(\overset{g}{R}ic, \overset{g}{\mathcal{R}}) = Q(g, \overset{g}{\mathcal{R}}) = Q(g, \overset{g}{\mathcal{C}}),
		\end{equation*}
		also holds.
	\end{corollary}

	\section{Perfect fluid space-time}\label{section5}
	
	A Lorentzian manifold is called a \textit{perfect fluid space-time} if the Ricci tensor $\overset{g}{R}ic$ has the form  
	\begin{equation}\label{eq:PF}
		\overset{g}{R}ic=ag+b\pi\otimes\pi,
	\end{equation}
	where $a$ and $b$ are scalars. Every RW space-time is a perfect fluid space-time \cite{neill1983}, while the converse has been explored in  \cite{dede2023,dede2024}. For dimension $n=4$, a GRW space-time is a perfect fluid space-time if and only if it is a RW space-time \cite{gutierrez2009}. In the geometric literature, a perfect fluid space-time is known as a \textit{quasi-Einstein manifold} (which may have a metric of arbitrary signature).
	
	It was shown in \cite{ucd2024} that a Lorentzian manifold with a semi-symmetric metric connection and a unit timelike torse-forming vector, whose curvature tensor vanishes, is a perfect fluid space-time (see Theorem 1.3 in \cite{ucd2024}).
	
	Here, we investigate the application of a semi-symmetric metric $P$-connection to a perfect fluid space-time.  Ricci solitons on Lorentzian manifolds with a semi-symmetric metric $P$-connection were studied in \cite{li2023Ln}, where, among other results, the following relation was established for perfect fluid space-times with a semi-symmetric metric $P$-connection
	\begin{equation}\label{eq:a-b}
		a-b=n-1.
	\end{equation}
	Furthermore, the Ricci tensor of such a space-time can be written in the form 
	\begin{equation}\label{eq:PF2}
		\overset{g}{R}ic = \left(\frac{\overset{g}{r}}{n-1} - 1\right) g + \left(\frac{\overset{g}{r}}{n-1} - n \right)\pi\otimes\pi.
	\end{equation}
	In general, the scalar curvature of a perfect fluid space-time with a semi-symmetric metric $P$-connection is not constant (see Lemma 3. in \cite{li2023Ln}). 
	
	It was demonstrated in \cite{chaki2000} that quasi-Einstein manifolds are not, in general, Ricci semi-symmetric. Every three-dimensional quasi-Einstein manifold is pseudo-symmetric \cite{deszcz1994}. We now prove the following theorem in an $n$-dimensional perfect fluid space-time with a semi-symmetric metric $P$-connection.
	
	
	\begin{theorem}\label{thm:Ricipseudosimetricankonst}
		A perfect fluid space-time $(\mathcal{M},g,\overset{1}{\nabla})$ is a Ricci pseudo-symmetric manifold of constant type which satisfies the relation
		\begin{equation*}
			\overset{g}{\mathcal{R}}\cdot \overset{g}{R}ic = Q(g, \overset{g}{R}ic).
		\end{equation*} 
	\end{theorem}
	\begin{proof}
		In a perfect fluid space-time (\ref{eq:PF}) we have
		\begin{equation*}
			\begin{split}
				(\overset{g}{R} (X,Y)\cdot \overset{g}{R}ic)(U,V) & =  - \overset{g}{R}ic(\overset{g}{R} (X,Y)U,V) - \overset{g}{R}ic(U,\overset{g}{R} (X,Y)V) \\
				& = - b \pi(\overset{g}{R} (X,Y)U)\pi(V) - b \pi(U)\pi(\overset{g}{R} (X,Y)V)
			\end{split}
		\end{equation*}
		where we used the anti-symmetry of the Riemannian curvature tensor, i.e. $\overset{g}{\mathcal{R}} (X,Y,U,V) = - \overset{g}{\mathcal{R}} (X,Y, V,U)$. Considering equation (\ref{eq:piRXYZ}), we obtain
		\begin{equation}\label{eq:RputaRicPF}
			\begin{split}
				(\overset{g}{R} (X,Y)\cdot \overset{g}{R}ic)(U,V)  =  b(&-g(Y,U)\pi(X)\pi(V)  + g(X,U)\pi(Y)\pi(V) \\
				&	- g(Y,V)\pi(X)\pi(U) + g(X,V)\pi(Y)\pi(U) ).
			\end{split}
		\end{equation}
		On the other hand, in a perfect fluid space-time (\ref{eq:PF}) we have
		\begin{equation*}
			\begin{split}
				Q(g, \overset{g}{R}ic)(U,V;X,Y)  = & -g(Y,U)\overset{g}{R}ic(X,V)  + g(X,U)\overset{g}{R}ic(Y,V) \\ 
				& - g(Y,V)\overset{g}{R}ic(X,U) + g(X,V)\overset{g}{R}ic(Y,U) \\
				= &  b(-g(Y,U)\pi(X)\pi(V)  + g(X,U)\pi(Y)\pi(V) \\
				& \quad - g(Y,V)\pi(X)\pi(U) + g(X,V)\pi(Y)\pi(U) ).
			\end{split}
		\end{equation*}
		The right-hand sides of the last two equations are equal, i.e. $\overset{g}{\mathcal{R}}\cdot \overset{g}{R}ic = Q(g, \overset{g}{R}ic)$.
	\end{proof}

	From equation (\ref{eq:PF2}), relation (\ref{eq:RputaRicPF}) can be written as
	\begin{equation*}
		\overset{g}{\mathcal{R}}\cdot \overset{g}{R}ic = \left(\frac{\overset{g}{r}}{n-1} - n \right)Q(g,\Pi),
	\end{equation*} 
	where $\Pi=\pi\otimes\pi$. Note that $Q(g,\Pi)=0$ cannot hold, because it would imply $\pi=0$, which is impossible. Combining the previous theorem and equation (\ref{eq:R1putaRic1}) yields the following corollary.
	\begin{corollary}
		A perfect fluid space-time $(\mathcal{M},g,\overset{1}{\nabla})$ satisfies the relation	$\overset{1}{\mathcal{R}}\cdot \overset{1}{R}ic=0$.
	\end{corollary}
	
	Since the Ricci tensors $\overset{0}{R}ic$, $\overset{4}{R}ic$, $\overset{5}{R}ic$ are non-zero (see Theorem \ref{thm:Ricirazlicitiodnule}), we now consider slightly weaker conditions, defining special classes of GRW space-times $(\mathcal{M},g,\overset{1}{\nabla})$.
	\begin{definition}\label{def:pftheta}
		A GRW space-time $(\mathcal{M},g,\overset{1}{\nabla})$ is a perfect fluid space-time of the $\theta$-th kind, $\theta=0,1,4,5$ if
		\begin{equation*}
			\overset{\theta}{R}ic=\overset{\theta}{a}g+\overset{\theta}{b}\pi\otimes\pi, \;\;\; \theta=0,1,4,5,
		\end{equation*}
		where $\overset{\theta}{a}$, $\overset{\theta}{b}$ are smooth functions.
	\end{definition}
	
	Contracting the previous equation yields
	\begin{equation*}
		\overset{\theta}{r}=\overset{\theta}{a}n+\overset{\theta}{b}, \;\;\; \theta=0,1,4,5.
	\end{equation*}
	
	Using the properties of the Ricci tensors $\overset{\theta}{R}ic$, $\theta=0,1,4,5$, we can obtain expressions for $\overset{\theta}{a}$, $\overset{\theta}{b}$, so the perfect fluid space-times of the $\theta$-th kind can be written in the following forms
	\begin{align*}
		\overset{0}{R}ic & = \left(\frac{\overset{0}{r}}{n-1} - \frac{1}{4} \right) g + \left(\frac{\overset{0}{r}}{n-1} - \frac{n}{4} \right)\pi\otimes\pi, \\
		\overset{1}{R}ic & =\frac{\overset{1}{r}}{n-1}(g+\pi\otimes\pi), \\
		\overset{4}{R}ic & = \left(\frac{\overset{4}{r}}{n-1} - 1 \right) g + \left(\frac{\overset{4}{r}}{n-1} - n \right) \pi\otimes\pi,\\
		\overset{5}{R}ic & = \left(\frac{\overset{5}{r}}{n-1} - \frac{1}{2} \right) g + \left(\frac{\overset{5}{r}}{n-1} - \frac{n}{2} \right)\pi\otimes\pi.
	\end{align*}
	
	The following theorem demonstrates that the perfect fluid space-times of the $\theta$-th kind coincide with the perfect fluid space-time (with respect to the metric $g$).	
	\begin{theorem}
		A GRW space-time $(\mathcal{M},g,\overset{1}{\nabla})$ is a perfect fluid space-time if and only if it is a perfect fluid space-time of the $\theta$-th kind, $\theta=0,1,4,5$.
	\end{theorem}
	\begin{proof}
		We provide the proof for $\theta=0$. If a GRW space-time $(\mathcal{M},g,\overset{1}{\nabla})$ is a perfect fluid space-time, then equation (\ref{eq:PF2}) holds. Substituting this equation into (\ref{eq:Ric0cssLn}) yields
		\begin{equation*}
			\overset{0}{R}ic = \left(\frac{\overset{g}{r}}{n-1} -n \right) g + \left(\frac{\overset{g}{r}}{n-1} - \frac{5n-1}{4} \right)\pi\otimes\pi.
		\end{equation*}
		Using equation (\ref{eq:r0css}), i.e. $4\overset{0}{r} = 4\overset{g}{r}-(n-1)(4n-1)$, the previous equation we can rewrite as
		\begin{equation}\label{eq:Ric0PF}
			\overset{0}{R}ic  = \left(\frac{\overset{0}{r}}{n-1} - \frac{1}{4} \right) g + \left(\frac{\overset{0}{r}}{n-1} - \frac{n}{4} \right)\pi\otimes\pi,
		\end{equation}
		providing that $(\mathcal{M},g,\overset{1}{\nabla})$ is a perfect fluid space-time of the zeroth kind.
		
		Conversely, if (\ref{eq:Ric0PF}) holds, then substituting this equation into (\ref{eq:Ric0cssLn}) and using (\ref{eq:r0css}) gives (\ref{eq:PF2}).
	\end{proof}
	
 \begin{remark}
			The Definition \ref{def:pftheta} does not apply to any non-symmetric connections, because in the general case the Ricci tensors $\overset{\theta}{R}ic$ associated with such connections are not symmetric. In that case, one could take $sym\overset{\theta}{R}ic$ instead of $\overset{\theta}{R}ic$, where
			\begin{equation*}
				sym\overset{\theta}{R}ic(X,Y)=\frac{1}{2}(\overset{\theta}{R}ic(X,Y)+\overset{\theta}{R}ic(Y,X)).
			\end{equation*}
			For such manifolds with respect to the semi-symmetric metric connection you can see paper \cite{han2021}. The relation $\overset{\theta}{R}ic =sym \overset{\theta}{R}ic$ is valid for a concircularly semi-symmetric metric connection (and for a semi-symmetric metric $P$-connection), which justifies Definition \ref{def:pftheta}. 
	\end{remark}
	
	\section{Application to the theory of relativity}\label{section6}
	
	To apply the preceding results to the theory of relativity, we now consider a $GRW$ space-time  $(\mathcal{M},g,\overset{1}{\nabla})$ in dimension $n=4$. \textit{Einstein's field equations} are fundamental equations in the theory of relativity, connecting the curvature of a space-time to the mass and energy of matter. Here we focus on Einstein's field equations without the cosmological constant which read
	\begin{equation}\label{eq:EFE}
		\overset{g}{R}ic - \frac{\overset{g}{r}}{2} g = k\tau,
	\end{equation}
	where $\tau$ is the \textit{energy-momentum tensor} (of type $(0,2)$) and $k$ is the \textit{gravitational constant}.
	
	Space-times with a semi-symmetric energy-momentum tensor of the form $\overset{g}{\mathcal{R}}\cdot \tau=0$ were studied in \cite{ucdv2015}, while those
	with a pseudo-symmetric energy-momentum tensor of the form $\overset{g}{\mathcal{R}}\cdot \tau = f Q(g, \tau)$ were investigated in \cite{mallickde2016}. The following theorem was proven.     
	\begin{theorem}\cite{mallickde2016}
		A general relativistic space-time with a pseudo-symmetric energy-momentum tensor is Ricci pseudo-symmetric and vice-versa.
	\end{theorem}
	This theorem directly implies the following statement.
	\begin{corollary}\cite{ucdv2015}
		A general relativistic space-time with a semi-symmetric energy-momentum tensor is Ricci semi-symmetric and vice-versa.
	\end{corollary}
	
	
	We now prove the next theorem.
	
	\begin{theorem}
		Let  $(\mathcal{M},g,\overset{1}{\nabla})$ be a $GRW$ space-time which satisfies Einstein's field equations without the cosmological constant. Then
		\begin{equation}\label{eq:Rputatau}
			\overset{\theta}{\mathcal{R}}\cdot \tau = f Q(g, \tau), \;\; \theta=0,1,4,
		\end{equation} 
		holds if and only if 
		\begin{equation}\label{eq:RputaRicg}
			\overset{\theta}{\mathcal{R}}\cdot \overset{g}{R}ic = f Q(g, \overset{g}{R}ic),
		\end{equation} 
		holds, where $f$ is an arbitrary function.
	\end{theorem}
	\begin{proof}
		Equation (\ref{eq:Rputatau}) implies
		\begin{equation*}
			\begin{split}
				-\tau(\overset{\theta}{R} (X,Y)U,V) - \tau(U,\overset{\theta}{R} (X,Y)V)=f(&-g(Y,U)\tau(X,V)  +  g(X,U)\tau(Y,V)  \\ 
				& - g(Y,V)\tau(U,X) + g(X,V)\tau(U,Y)).
			\end{split}
		\end{equation*}
		Using (\ref{eq:EFE}), we obtain
		\begin{equation*}
			\begin{split}
				&-\overset{g}{R}ic(\overset{\theta}{R} (X,Y)U,V) + \frac{\overset{g}{r}}{2} g(\overset{\theta}{R} (X,Y)U,V) - \overset{g}{R}ic(U,\overset{\theta}{R} (X,Y)V) + \frac{\overset{g}{r}}{2} g(U,\overset{\theta}{R} (X,Y)V) \\ & =f(-g(Y,U)\overset{g}{R}ic(X,V)  +  g(X,U)\overset{g}{R}ic(Y,V)  - g(Y,V)\overset{g}{R}ic(U,X) + g(X,V)\overset{g}{R}ic(U,Y)).
			\end{split}
		\end{equation*}
		Since $\overset{\theta}{\mathcal{R}} (X,Y,U,V)=- \overset{\theta}{\mathcal{R}} (X,Y,V,U)$, $\theta=0,1,4$ (see\cite{mincic1999}), the last equation gives (\ref{eq:RputaRicg}).
	\end{proof}
	
	Based on Theorem \ref{thm:Ricipseudosimetricankonst} we can prove the following statement. 
	\begin{theorem}\label{thm:PFRputatau}
		Let $(\mathcal{M},g,\overset{1}{\nabla})$ be a perfect fluid space-time which satisfies Einstein's field equations without the cosmological constant. Then the relation
		\begin{equation*}
			\overset{g}{\mathcal{R}}\cdot \tau = Q(g, \tau)
		\end{equation*} 
		holds.
	\end{theorem}
	
	The energy-momentum tensor for a perfect fluid space-time has the form
	\begin{equation}\label{eq:energy-momentum}
		\tau = p g + (\sigma+p) \pi \otimes \pi,
	\end{equation}
	where $\sigma$ is the \textit{energy density} and $p$ is the \textit{isotropic pressure}, where $\sigma+p\ne 0$ and $\sigma>0$ (see pp. 61-63 in \cite{duggal1999} or pp. 61. in \cite{stephani2009}). The relationship between pressure and energy density defines the \textit{equation of state}. The equation of state for \textit{dark energy} can be described by the relation $\frac{p}{\sigma}=\varpi$, for $\varpi<-\frac{1}{3}$, while for $\varpi<-1$ we have \textit{phantom dark energy} (for example, see \cite{bhar2021}). If $\sigma=p$, then we have a \textit{stiff matter fluid} (pp. 66. in \cite{stephani2009}).  It was shown in \cite{mantica2016} that an $n$-dimensional Lorentzian manifold ($n\geq 3$) whose Ricci tensor has the form $\overset{g}{R}ic=-\overset{g}{r}\pi\otimes\pi$ represents a stiff matter fluid. This motivates the following theorem.

	\begin{theorem}
		A GRW space-time $(\mathcal{M},g,\overset{1}{\nabla})$ satisfying relation $ \overset{0}{\mathcal{R}}\cdot \overset{0}{R}ic =0$ is a stiff matter fluid with respect to $\overset{0}{\nabla}$,  where $\overset{0}{\nabla}$ is associated symmetric connection of $\overset{1}{\nabla}$.
	\end{theorem}
	\begin{proof}
		If $ \overset{0}{\mathcal{R}}\cdot \overset{0}{R}ic =0$ holds, then the observed manifold is Einstein (see Theorem \ref{thm:RalphaputaRic=0}) satisfying relation (\ref{eq:Ajnstajn}). Substituting the equation (\ref{eq:Ajnstajn}) into (\ref{eq:Ric0cssLn}) yields
		\begin{equation*}
			\overset{0}{R}ic=-\frac{3}{4} \pi\otimes\pi. 
		\end{equation*}
		This implies $\overset{0}{r}=\frac{3}{4}$, so the previous equation can be written in the form
		\begin{equation*}
			\overset{0}{R}ic=-\overset{0}{r}\pi\otimes\pi.
		\end{equation*}
		Thus, the GRW space-time $(\mathcal{M},g,\overset{1}{\nabla})$  is a stiff matter fluid with respect to the connection $\overset{0}{\nabla}$.
	\end{proof}
	
	Based on equations (\ref{eq:EFE}) and (\ref{eq:energy-momentum}) we get
	\begin{equation*}
		\overset{g}{R}ic= \frac{1}{2}(\overset{g}{r} + 2k p)g+ k(\sigma+p) \pi\otimes\pi.
	\end{equation*}
	Contracting the previous equation gives the scalar curvature $\overset{g}{r} = k(\sigma -3p)$, so the Ricci tensor of a perfect fluid space-time has the form
	\begin{equation}\label{eq:RicPF}
		\overset{g}{R}ic=\frac{k(\sigma-p)}{2}g+k(\sigma+p)\pi\otimes\pi.
	\end{equation}
	
	Since $a-b=3$ holds in a four-dimensional perfect fluid space-time with a semi-symmetric metric $P$-connection (see equation (\ref{eq:a-b})), we obtain
	\begin{equation}\label{eq:sigma+3rho<0}
		k(\sigma + 3p)=-6.
	\end{equation}
	Considering that $k>0$, it follows
	\begin{equation*}
		\sigma + 3p<0.
	\end{equation*}
	This means that the \textit{strong energy condition} is violated, for which $\sigma + 3p\geq 0$ and $\sigma + p\geq 0$ should apply.
	
	\begin{theorem}\label{thm:SEC}
		In a perfect fluid space-time $(\mathcal{M},g,\overset{1}{\nabla})$ which satisfies the Einstein's field equations without the cosmological constant, the strong energy condition is violated.
	\end{theorem}
	

	\begin{remark}
		Theorem \ref{thm:SEC} can be considered as a consequence of Theorem \ref{thm:Ricipseudosimetricankonst} herein and Theorem II. 2.(ii) from \cite{mallickde2016}.
	\end{remark}
	
	\begin{remark}
		Equation (\ref{eq:sigma+3rho<0}) was also derived in \cite{yilmaz2023} for a pseudo $Z$-symmetric space-time.
	\end{remark}

	We now prove the following theorems.

	\begin{theorem}
		If a perfect fluid space-time $(\mathcal{M},g,\overset{1}{\nabla})$ which satisfies the Einstein's field equations without the cosmological constant is a Ricci semi-symmetric manifold, then the equation of state represents a phantom barrier.
	\end{theorem}
	\begin{proof}
		If a perfect fluid space-time $(\mathcal{M},g,\overset{1}{\nabla})$ is Ricci semi-symmetric, then for $n=4$ the Ricci tensor has the form
		\begin{equation*}
			\overset{g}{R}ic = 3g,
		\end{equation*}
		where we take into consideration Theorem \ref{thm:GRWRicipolu-sim}. Comparing the previous equation with (\ref{eq:RicPF}) yields $k(\sigma+p)=0$, i.e. $\sigma+p=0$.
	\end{proof}

	\begin{theorem}
		If, in a perfect fluid space-time $(\mathcal{M},g,\overset{1}{\nabla})$ which satisfies the Einstein's field equations without the cosmological constant, either $ \overset{\alpha}{\mathcal{R}}\cdot \overset{\alpha}{R}ic=0$, $\alpha=0,4$, or $ \overset{5}{\mathcal{R}}\cdot \overset{5}{R}ic $ $=\overset{g}{\mathcal{R}}\cdot \overset{g}{R}ic$ holds, then the equation of state represents a phantom barrier.
	\end{theorem}

	\section{Discussion and conclusions}

		Following the publication of paper  \cite{chen2014}, the investigation of GRW space-times through various vector fields and geometric structures has become a significant area of research. In this paper, we have shown that a Lorentzian manifold equipped with a concircularly semi-symmetric metric connection, with a unit timelike vector, reduces to a GRW space-time. Specifically, we demonstrated that the generator of the mentioned connection becomes a unit timelike torse-forming vector (\ref{eq:Ptorse-forming}). The idea of the authors was to connect the generator of the mentioned connection with the vector on the basis of which the necessary and sufficient conditions for GRW space-times were obtained. 
		Studying Lorentzian manifolds equipped with a concircularly semi-symmetric metric connection is effectively equivalent to studying such manifolds that admit a unit time-like torse-forming vector field, as defined in equation (\ref{eq:Ptorse-forming}), with respect to the Levi-Civita connection. In this context, it is not necessary to explicitly reference the concircularly semi-symmetric metric connection or the semi-symmetric metric $P$-connection, since the same geometric structures and results can be obtained solely through the Levi-Civita connection and the associated unit time-like torse-forming vector field (\ref{eq:Ptorse-forming}). However, this equivalence does not extend to the case of the semi-symmetric metric connection in general, as its defining vector field need not be torse-forming.
		
		Given the numerous literature on vector fields with related properties on Lorentzian manifolds, further exploration of the geometric and physical characteristics of these manifolds is warranted. Future work may deal with the perfect fluid space-times with this connection, including their application to general relativity, considering both the cases of Einstein's field equations with and without the cosmological constant. 
		A promising direction for future research is to investigate the observed connection within the framework of the theory of relativity, focusing on formulations that involve the torsion tensor. In particular, one could examine Einstein’s field equations expressed in terms of the Ricci tensor relative to a non-symmetric connection, following the approaches in \cite{kranas2019,csillag2024}. Building on \cite{csillag2024} which analyzed generalized Einstein's field equations incorporating the symmetric part of the Ricci tensor for a semi-symmetric metric connection, it would be of interest to begin with a concircularly semi-symmetric metric connection. This approach would lead to a simpler form of the generalized Einstein’s field equations and recover a special case of the results presented in \cite{csillag2024}.

	\section{Acknowledgement}
	
	The authors thanks the anonymous referees for their suggestions and comments that helped to improve the paper.
	
	Miroslav D. Maksimovi\'c was partially supported by the Ministry of Education, Science and Technological Development of the Republic of Serbia, project no. 451-03-137/2025-03/200123, by the project of Faculty of Sciences and Mathematics, University of Pri\v stina in Kosovska Mitrovica (internal-junior project IJ-2303) and by the Bulgarian Ministry of Education and Science, Scientific Programme "Enhancing the Research Capacity in Mathematical Sciences (PIKOM)", No. DO1-67/05.05.2022. Milan Lj. Zlatanović was partially supported by a grant from the IMU-CDC and the Simons Foundation, by the Ministry of Education, Science and Technological Development of the Republic of Serbia (project no. 451-03-137/2025-03/200124), and by the Bulgarian Ministry of Education and Science, Scientific Programme “Enhancing the Research Capacity in Mathematical Sciences (PIKOM)” (No. DO1-67/05.05.2022). M.Z. also thanks IMI-BAS, Sofia, for providing support and excellent research conditions during the final stage of preparing the paper.

	Miroslav D. Maksimovi\'c \\
	{University of Pri\v stina in Kosovska Mitrovica, Faculty of Sciences and Mathematics, Department of Mathematics, Kosovska Mitrovica, Serbia, and	\\
		Institute of Mathematics and Informatics, Bulgarian Academy of Sciences, Sofia 1113, Acad. G. Bonchev Str., Bl. 8, Bulgaria}
	\\
	email: miroslav.maksimovic@pr.ac.rs
	\\
	
	Milan Lj. Zlatanovi\'c \\
	{University of Ni\v s, Faculty of Sciences and Mathematics, Department of Mathematics, Ni\v s, Serbia,}
	\\
	email: zlatmilan@yahoo.com
	\\
	
	Milica R. Vu\v{c}urovi\'c \\
	{University of Pri\v stina in Kosovska Mitrovica, Faculty of Sciences and Mathematics, Department of Mathematics, Kosovska Mitrovica, Serbia,}
	\\
	email: milica.vucurovic@pr.ac.rs


\begin{thebibliography}{1}
		
		
		
		\small{ 
			
			\bibitem{alias1995} L. J. Alias, A. Romero, M. Sanchez, \textit{Uniqueness of complete spacelike hypersurfaces of constant mean curvature in Generalized Robertson-Walker space-times}, Gen. Relativity Gravitation, 27(1), 71–84, (1995)
			\bibitem{manticamolinari2017} C. A. Mantica, L. G. Molinari, \textit{Generalized Robertson-Walker spacetimes: A survey}, Int. J. Geom. Methods Mod. Phys. 14 (3), 1730001, (2017).
			\bibitem{yano1944} K. Yano, \textit{On the torse-forming direction in a Riemannian spaces}, Proc. Imp. Acad. Tokyo, 20 (6), 340-345, (1944).
			\bibitem{chaubey2022} S. K. Chaubey, Y. J. Suh, \textit{Riemannian concircular structure manifolds}, Filomat, 36 (19), 6699-6711, (2022).
			\bibitem{chen2014} B. Y. Chen, \textit{A simple characterization of generalized Robertson-Walker spacetimes}, Gen. Relativ. Gravit. 46(12), 1833, (2014).
			\bibitem{defever1994} F. Defever, R. Deszcz, L. Verstraelen, L. Vrancken, \textit{On pseudosymmetric space-times}, Jornal of Mathematical Physics, 35, 5908, (1994).
			\bibitem{arslan2014} K. Arslan, R. Deszcz, R. Ezentas, M. Hotlos, C. Murathan, \textit{On generalized Robertson-Walker spacetimes satisfying some curvature condition}, Turk. J. Math. 38(2), 353-373, (2014).
			\bibitem{chaki1987} M. C. Chaki, \textit{On pseudo symmetric manifolds}, An. Stiint. Univ. Al. I. Cuza Iasi. Mat. (N. S.), 33, 53-58, (1987). 
			\bibitem{shaikh2015} A. A. Shaikh, R. Deszcz, M. Hotlos, J. Jelowicki, \textit{On pseudosymmetric manifolds}, Publ. Math. Debrecen, 86 (3-4) (2016) 433-456.	
			\bibitem{deszcz1989} R. Deszcz, \textit{On Ricci-pseudosymmetric warped products}, Demonstr. Math., 22, 1053-1065, (1989).
			
			
			\bibitem{deszcz2023} R. Deszcz, M. Glogowska, M. Hotlos, M. Petrovi\'c-Torga\v{s}ev, G. Zafindratafa, \textit{On semi-Riemannian manifolds satisfying some generalized Einstein metric conditions}, International Electronic Journal of Geometry, 16(2), 539-576, (2023).
		 \bibitem{cartan1923}  E. Cartan, \textit{Sur les vari\'et\'es \`{a} connexion affine et la th\'eorie de la relativ\'e g\'en\'eralis\'ee. Part I}, Ann. Ec. Norm., 40, 325-412, (1923).
				\bibitem{friedmann1924} {A. Friedmann, J. A. Schouten,} {\it  \"{U}ber die geometrie der halbsymmetrischen \"{U}bertragung},  Math. Zeitschr., 21, 211--223, (1924).
				\bibitem{agricola2016} I. Agricola, M. Kraus, \textit{Manifolds with vectorial torsion}, Dif. Geo. App., 45, 130-147, (2016).
				\bibitem{lehel2024} L. Csillag, R. Hama, M. J\'ozsa, T. Harko, S. V. Sabau, \textit{Length-preserving biconnection gravity and its cosmological implications}, J. Cosm. Astro. Phys., 12, 024, (2024).
			
			\bibitem{li2023Ln} Y. Li, H. A. Kumara, M. S. Siddesha, D. M. Naik, \textit{Charachterization of Ricci almost solitons on Lorentzian manifolds}, Symmetry, 15, 1175, (2023).
			\bibitem{ucd2024} U. C. De, K. De, S. G\"uler, \textit{Characterizations of a Lorentzian manifold with a semi-symmetric metric connection}, Publicationes Mathematicae Debrecen, 104(3-4), 329-341, (2024).
			\bibitem{yilmaz2023} H. B. Yilmaz, \textit{On pseudo Z-symmetric Lorentzian manifolds admitting a type of semi-symmetric metric connection}, Anal. Math. Phys. 13, 18, (2023).
			\bibitem{chaubeysuhde2020} S. K. Chaubey, Y. J. Suh, U. C.  De, \textit{Characterizations of the Lorentzian manifolds admitting a type of semi-symmetric metric connection}, Anal. Math. Phys. 10, 61, (2020). 
			
			\bibitem{chaubey2021} S. K. Chaubey, U. C. De, M. D. Siddiqi, \textit{Characterization of Lorentzian manifolds with a semi-symmetric linear connection}, J. Geom. Phys. 166, 104269, (2021).
			
			
			\bibitem{chaubey2022b} S. K. Chaubey, Y. J. Suh, \textit{Characterizations of Lorentzian manifolds}, J. Math. Phys., 63, 062501, (2022).
			
			
			\bibitem{li2024Ln} Y. Li,  M. S. Siddesha, H. A. Kumara, M. M. Praveena, \textit{Charachterization of Bach and Cotton tensors on a class of Lorentzian manifolds}, Mathematics, 12, 3130, (2024).
			\bibitem{suh2024} Y. J. Suh, S. K. Chaubey, M. N. I. Khan, \textit{Lorentzian manifolds: A characterization with a type of semi-symmetric non-metric connection},  Rev. Math. Phys. 36(03), 2450001, (2024).
			
		 	\bibitem{ivanov2010}{S. Ivanov,} {\it Heterotic supersymmetry, anomaly cancellation and equations of motion}, Phys. Lett. B,  685(2-3), (2010), 190-196. 
				
				\bibitem{iosifidis2024} D. Iosifidis, \textit{On a torsion/curvature analogue of dual connections and statistical manifolds}, J. Geo. Phys., 196, 105064, (2024).
				\bibitem{graiff1952} F. Graiff, \textit{Sulla possibilita di costruire parallelogrammi chiusi in alcune varieta a torsione}, Boll. d. Un. Mat. Ital. Ser III, 7, 132-135, (1952).
			
			\bibitem{pak1969} E. Pak, \textit{On the pseudo-Riemannian spaces}, J. Korean Math. Soc., 6, 23-31, (1969). 
			\bibitem{zlatanovic2021} M. Petrovi\'c, N. Vesi\'c, M. Zlatanovi\'c, \textit{Curvature properties of metric and semi-symmetric linear connections}, Quaest. Math. 45(10) (2022) 1603-1627.
			
			
			
			
			
			
			\bibitem{slosarska1984} W. Slosarska, \textit{On some invariants of Riemannian manifold admitting a concircularly semi-symmetric metric connection}, Demonstr. Math., 17 (1), 251-257 (1984). 
			\bibitem{yano1940} K. Yano, \textit{Concircular geometry I. Concircular transformations}, Proc. Imp. Acad. Tokyo, 16 (6), 195-200, (1940). 
			
			\bibitem{chaubey2019} S. K. Chaubey, J. W. Lee, S. K. Yadav, \textit{Riemannian manifolds with a semi-symmetric metric P-connection}, J. Korean Math. Soc., 56 (4), 1113-1129, (2019). 
			\bibitem{mpvz2023} M. Maksimovi\'c, M. Petrovi\'c, N. Vesi\'c, M. Zlatanovi\'c, \textit{Concircularly semi-symmetric metric connection}, Quaest. Math., 47(3), 557-576, (2024).  
			
			\bibitem{siddiqi2019} M. D. Siddiqi, \textit{Ricci $\rho$-soliton and geometrical structure in a dust fluid and viscous fluid spacetime}, Bulg. J. Phys. 46, 163-173, (2019).
			\bibitem{deszcz2001}	R. Deszcz, M. Hotlos, Z. Sent\"urk, \textit{On some family of generalized Einstein metric conditions}, Demonstratio Math. 34(4), 943–954, (2001).
			\bibitem{neill1983} B. O'Neill, \textit{Semi-Riemannian geometry with applications to the relativity}, Academic Press, New York-London, (1983).
			\bibitem{dede2023} K. De, U. C. De, \textit{Perfect ﬂuid spacetimes obeying certain restrictions on the energy-momentum tensor}, Filomat, 37(11), 3483-3492, (2023).
			\bibitem{dede2024} K. De, U. C. De, Lj. Velimirovi\'c, \textit{Some curvature properties of perfect fluid spacetimes}, Quaestiones Mathematicae, 47(4), 751-764, (2024).
			
			
			\bibitem{gutierrez2009} M. Gutierrez, B. Olea, \textit{Global decomposition of a Lorentzian manifold as a generalized Robertson-Walker space}, Differ. Geom. Appl. 27, 146-156, (2009).
			\bibitem{chaki2000} M. C. Chaki, R. K. Maity, \textit{On quasi Einstein manifolds}, Publ. Math. Debrecen, 57 (3-4), 297-306, (2000).
			
			\bibitem{deszcz1994} R. Deszcz, L. Verstraelen, S. Yaprak, \textit{Warped products realizing a certain condition of pseudosymmetry type imposed on the Weyl curvature tensor}, Chinese J. Math., 22, 139-157, (1994). 
			
			
			
			\bibitem{han2021} Y. Han, A. De, P. Zhao, \textit{On a semi-quasi-Einstein manifold}, J. Geom. Phys., 155, 103739, (2021).
			
			\bibitem{ucdv2015} U. C. De, Lj. Velimirovi\'c, \textit{Spacetimes with semisymmetric energy-momentum tensor}, Int. J. Theor. Phys. 54, 1779-1783, (2015).
			\bibitem{mallickde2016} S. Mallick, U. C. De, \textit{Spacetimes with pseudosymmetric energy-momentum tensor}, Communications in Physics, 26(2), 121-128, (2016).
			
			\bibitem{mincic1999} S. Min\v{c}i\'c, \textit{Some characteristics of curvature tensors of nonsymmetric affine connexion},12th Yugoslav Geometric Seminar (Novi Sad, 1998), Novi Sad J. Math., 29(3), 169-186, (1999).
			
			
			\bibitem{duggal1999} K. L. Duggal, R. Sharma, \textit{Symmetries of spacetimes and Riemannian manifolds}, Mathematics and its Applications, Kluwer Academic Publishers, (1999).
			\bibitem{stephani2009} H. Stephani, D. Kramer, M. Mac-Callum, C. Hoenselaers, E. Hertl, \textit{Exact Solutions of Einstein's Field Equations}, 2nd edn., Cambridge Monographs on Mathematical Physics, Cambridge University Press, (2009).
			
			\bibitem{bhar2021} P. Bhar, \textit{Dark energy stars in Tolman–Kuchowicz spacetime in the context of Einstein gravity}, Physics of the Dark Universe, 34, 100879, (2021).
			
			\bibitem{mantica2016} C. A. Mantica, Y. J. Suh, U. C. De, \textit{A note on generalized Robertson–Walker space-times}, Int. J. Geom. Methods Mod. Phys. 13, 1650079, (2016).
			
				
				\bibitem{kranas2019} D. Kranas, C. G. Tsagas, J. D. Barrow, D. Iosifidis, \textit{Friedmann-like universes with torsion}, Eur. Phys. J. C, 79, 341, (2019).
				\bibitem{csillag2024} L. Csillag, T. Harko, \textit{Semi-symmetric metric gravity: From the Friedmann-Schouten geometry with torsion to dynamical dark energy models}, Phys. Dark Univ., 46, 101596, (2024). 
			
			
			
			
			
			
			
			
			
			
			
			
			
			
			
			
			
			
			
			
			
			
			
			
			
			
			
			
			
			
			
			
			
			
			
			
			
			
			
			
			
			
			
			
			
			
			
			
			
			
			
			
			
			
			
			
			
			
			
			
			
			
			
			
			
			
			
			
			
			
			
			
			
			
			
			
			
			
			
			
			
			
			
			
			
			
			
			
			
			
			
			
			
			
			
			
			
			
			
			
			
			
			
			
			
			
			
			
			
			
			
			
			
			
			
			
			
			
			
			
			
			
			
			
			
			
			
			
			
			
			
			
			
			
		}
	\end{thebibliography}
\end{document}